\documentclass[a4paper,oneside,english]{article}
\usepackage[T1]{fontenc}
\usepackage[latin9]{inputenc}
\usepackage{babel}
\usepackage{amsmath}
\usepackage{amsthm}
\usepackage{amstext}
\usepackage{amssymb}
\usepackage{setspace}

\usepackage{graphicx} 
\usepackage{mathabx}
\usepackage{stmaryrd}

\numberwithin{equation}{section}
\numberwithin{figure}{section}
\newtheorem{thm}{Theorem}[section]
\newtheorem{cor}[thm]{Corollary}
\theoremstyle{definition}

\theoremstyle{plain}
\newtheorem{prop}[thm]{Proposition}
\theoremstyle{remark}
\newtheorem{rem}[thm]{Remark}
\theoremstyle{plain}

\theoremstyle{plain}
\newtheorem{assumptionA}{Assumption}
\theoremstyle{plain}

\newtheorem{lemma}[thm]{Lemma}
\theoremstyle{plain}

\newcommand{\E}{\mathbb{E}}
\usepackage{hyperref}

\makeatother

\begin{document}

\title{A law of large numbers for branching Markov processes by the ergodicity of ancestral lineages}
\author{Aline Marguet\footnote{Univ. Grenoble Alpes, Inria, 38000 Grenoble, France. E-mail: aline.marguet@inria.fr}}
\maketitle
\date{}
\begin{abstract}
We are interested in the dynamic of a structured branching population where the trait of each individual moves according to a Markov process. The rate of division of each individual is a function of its trait and when a branching event occurs, the trait of a descendant at birth depends on the trait of the mother. We prove a law of large numbers for the empirical distribution of ancestral trajectories. It ensures that the empirical measure converges to the mean value of the spine which is a time-inhomogeneous Markov process describing the trait of a typical individual along its ancestral lineage. Our approach relies on ergodicity arguments for this time-inhomogeneous Markov process. We apply this technique on the example of a size-structured population with exponential growth in varying environment. 
\end{abstract}
\paragraph*{A.M.S classification:}
60J80, 60F17, 60F25, 60J85, 92D25.

\section{Introduction}
We are interested in the asymptotic behavior of a continuous-time structured branching Markov process. Each individual in the population is characterized by a trait which follows a Markovian dynamic and which influences the branching events. This trait may describe the position of an individual, its size, the number of parasites inside a cell, etc. The purpose of this article is to prove a law of large numbers i.e. the convergence of the empirical measure to a deterministic limit. 

The law of large numbers has already been proved in many different cases. For the convergence in discrete time of the proportions of individuals with a certain type in the population, we refer to \cite{Athreya1998a,Athreya1998b} with respectively a discrete or continuous set of types. The generalisation of the law of large numbers to general branching Markov processes has been obtained by Asmussen and Hering in \cite{Asmussen1976} in both discrete and continuous time. Their proof relies on a specific decomposition of the first moment semigroup which applies to the case of branching diffusions. In the context of cellular aging, Guyon \cite{guyon2007limit} proved the convergence of the empirical measure for bifurcating Markov chains using the ergodicity of the spine. A generalization of those results to binary Galton-Watson processes can be found in \cite{Delmas2010}. For results in varying environment, we mention \cite{BH2015,bansaye2013}. In continuous-time, we refer to \cite{georgii2003supercritical} for asymptotic results in the case of a finite number of types, to \cite{Harris2014} for a strong law of large numbers in the case of local branching and to \cite{Ren2014} for central limit theorems. The specific case of branching diffusions, popularized by Asmussen and Hering \cite{Asmussen1976}, is adressed in \cite{EHK2010}. We also mention \cite{Englander2006,Englander2009} for the study of the case of superdiffusions. For nonlocal branching results in continuous-time, we refer to \cite{BT} for the study of the proportion of infected cells in a population, to \cite{bansaye2011limit} for the case of a general Markov branching process with a constant division rate and to \cite{cloez} for the convergence of an empirical measure in the general case. Some of those results rely on spectral theory. Here, we will follow another approach which requires no use of eigenelements as in \cite{bansaye2011limit} or \cite{guyon2007limit}. In particular, it can be applied to time-inhomogeneous dynamics. 

The question of the asymptotic behavior of structured branching processes appears in many different situations and in particular in the modeling of cell population dynamics. In this context, the law of large numbers is a key result for the construction of an estimating procedure for the parameters of the model. We refer to \cite{HO2016} for the estimation of the division rate in the case of an age-structured population.

In this article, we prove the convergence of the empirical measure for a class of general branching Markov processes, using spinal techniques. More precisely, we use the characterization of the trait along a typical ancestral lineage introduced in \cite{marguet2016}. We adapt the techniques of \cite{hairer2011yet} and we prove that under classical conditions \cite[Chapters 15, 16]{Meyn}, the semigroup of the auxiliary process, which is a time-inhomogeneous Markov process, is ergodic. Using this property, we prove a law of large numbers for the empirical distribution of ancestral trajectories. We also apply this technique to an example in varying environment where the law of large numbers result holds. 

We describe briefly the branching process $(Z_t,t\geq 0)$ and we refer to \cite{marguet2016} for its rigorous construction. We assume that individuals behave independently and that for each individual $u$ in the population:
\begin{itemize}
\item its trait $(X_t^u, t\geq 0)$ follows a Markov process on $\mathcal{X}$ with infinitesimal generator and domain $\left(\mathcal{G},\mathcal{D}(\mathcal{G})\right)$,
\item it dies at time $t$ at rate $B(t,X_t^u)$ and is replaced by $2$ individuals,
\item the trait of the two children are both distributed according to $Q(X_t^u,\cdot)$. 
\end{itemize}
\begin{rem}\label{rem:simplification}
Two remarks are in order:
\begin{enumerate}
\item For the sake of clarity, we consider only binary division but the model can easily be extended to a random number of descendants as in \cite{marguet2016}. The choice of equal marginal distribution for the traits at birth simplifies calculation but is not mandatory. 
\item The reason why we choose to make the time-dependence of the division rate explicit is twofold. First, it is the case in the example we choose to develop in the last section of this article in order to tackle environment changes. Second, it highlights the (possible) time-inhomogeneity of the measure-valued branching process $Z$. We emphasize that this case is covered by the study in \cite{marguet2016} where the trait lives on $\mathcal{X}=\mathcal{Y}\times \mathbb{R}_+$.
\end{enumerate}
\end{rem}

We focus on the empirical measure which describes the current state of the population
\begin{align*}
\frac{1}{N_t}\sum_{u\in V_t}\delta_{X_t^u},\ t\geq 0,
\end{align*}
where $V_t$ denotes the set of individuals alive at time $t$ and $N_t$ its cardinal. A crucial quantity for the study of this probability measure is the first moment semigroup applied to the constant function equal to $1$ given by
\begin{align*}
m(x,s,t):=\mathbb{E}\left[N_t\big|Z_s=\delta_x\right].
\end{align*}
It is the mean number of individuals in the population at time $t$ starting at time $s$ with a single individual with trait $x\in\mathcal{X}$. In fact, the behavior of the empirical measure is linked with the behavior of a uniformly chosen individual in the population and the mean number of individuals in the population. More precisely, we have the following result, referred to as a Many-to-One formula \cite[Theorem 3.1]{marguet2016}, which holds under Assumptions \ref{assu:debut} and \ref{assu:regularite} given below: for all non-negative measurable functions $F$ on the space of c\`adl\`ag processes, for all $0\leq s\leq t$ and $x_0\in\mathcal{X}$,
\begin{equation}\label{mtomesurable}
\E\left[\sum_{u\in V_{t}}F\left(X_{s}^{u},s\leq t\right)\big| Z_0=\delta_{x_0}\right]=m(x_0,0,t)\E\left[F\left(Y_{s}^{(t)},s\leq t\right)\big|Y_0^{(t)}=x_0\right],
\end{equation}
 where $\left(Y_{s}^{(t)}, s\leq t\right)$ is a time-inhomogeneous Markov process, called the auxiliary process, whose infinitesimal generators $\left(\mathcal{A}_{s}^{(t)}, s\leq t\right)$ are given for all suitable functions $f$ and $x\in\mathcal{X}$ by
\begin{align}\label{eq:gene_auxi}
\mathcal{A}_s^{(t)}f(x)=\frac{\mathcal{G}(m(\cdot,s,t)f)(x)-f(x)\mathcal{G}(m(\cdot,s,t))(x)}{m(x,s,t)}+2B(s,x)\int_{\mathcal{X}}(f(y)-f(x))\frac{m(y,s,t)}{m(x,s,t)}Q(x,dy).
\end{align}
The auxiliary process corresponds to the trait of a typical individual in the population \cite{marguet2016}. More precisely, the family of operators $(P_{r,s}^{(t)},0\leq r\leq s \leq t)$ defined for all measurable functions $f$ by
\begin{align*}
P_{r,s}^{(t)}f(x)=\frac{R_{r,s}(fm(\cdot,s,t))(x)}{m(x,r,t)},
\end{align*}
where $R_{r,s}f(x)=\mathbb{E}\left[\sum_{u\in V_s} f(X_s^u)|Z_r = \delta_x\right]$ forms a time-inhomogeneous semigroup (i.e. $P_{r,u}^{(t)}P_{u,s}^{(t)}=P_{r,s}^{(t)}$ for all $r\leq u\leq s\leq t$), which is the semigroup of the auxiliary process.
It can also be exhibited using a change of probability measure. Indeed, by Feynman-Kac's formula (see \cite[Section 1.3]{del2004feynman}), we have
\begin{align*}
P_{r,s}^{(t)}f(x)=m(x,r,t)^{-1}\mathbb{E}\left[e^{\int_r^sB(X_v)dv}m(X_s,s,t)f(X_s)\big|X_r=x\right],
\end{align*}
where $(X_s,r\leq s\leq t)$ is a Markov process with infinitesimal generator $\mathcal{M}$ given by
\begin{align*}
\mathcal{M}f(x)=\mathcal{G}f(x)+2B(x)\int_{\mathcal{X}}\left(f(y)-f(x)\right)Q(x,dy).
\end{align*}
Then, the change of probability measure given by the $\sigma(X_l,l\leq s)$-martingale
\begin{align*}
M_s^{(t)}:=\frac{e^{\int_r^sB(X_s)ds}m(X_s,s,t)}{m(x,r,t)},\quad \text{for }r\leq s\leq t
\end{align*}
exhibits the probability measure corresponding to the auxiliary process.

The auxiliary process and its asymptotic behavior are the keys to obtain the main result of this article which is the following law of large numbers for the empirical distribution of ancestral trajectories: 

\begin{align*}
\left(\frac{\sum_{u\in V_{t+T}}F\left(X_{t+s}^u, s\leq T\right)}{N_{t+T}}-\mathbb{E}\left[F\left(Y_{t+s}^{(t+T)},s\leq T\right)\middle| Y_0^{(t+T)} = x_1\right]\right)\xrightarrow[t\rightarrow +\infty]{} 0\text{, in }\mathbb{L}_2(\delta_{x_0}),
\end{align*}
for all $x_0,x_1\in\mathcal{X}$ and $T>0$, where the $\mathbb{L}_2(\delta_{x_0})$-convergence is the $\mathbb{L}_2$-convergence with initial measure $\delta_{x_0}$.

This result ensures that the behavior of the whole population becomes deterministic asymptotically and that this behavior is given by the limit behavior of the auxiliary process. This weak law of large numbers gives information on the ancestral lineages in the population. To establish this result, we prove in particular that under the classical drift and minorization conditions \cite[Chapters 15, 16]{Meyn} adapted to the time-inhomogeneous case, the auxiliary process is ergodic in the sense that there exists $\overline{c}>0$ such that for all $x,y\in\mathcal{X}$, $T>0$, for all bounded measurable functions $F:\mathbb{D}([0,T],\mathcal{X})\rightarrow \mathbb{R}$ and all $0\leq r\leq t$, we have
\begin{align*}
\left|P_{r,t,T}F(x)-P_{r,t,T}F(y)\right|\leq C e^{-\overline{c}(t-r)}d(x,y)\left\Vert F\right\Vert_{\infty},
\end{align*}
where $d$ is a distance on $\mathcal{X}$, $C$ is a positive constant and
\begin{align}\label{eq:semigroup_traj}
P_{r,t,T}F(x):=\mathbb{E}\left[F\left(Y_{t+s}^{(t+T)}, s\leq T\right)\big| Y_r^{(t+T)}=x\right].
\end{align}

We also apply our method to study a size-structured population with a division rate that depends both on the trait and the time. This example models the dynamic of size-structured cell population. Hence, the trait of interest is the size of each individual, increasing exponentially at rate $a$. We assume that each cell divides at rate $B(t,x)=x\varphi(t)$, where $\varphi$ is a positive function which describes environment changes. At division, a cell of size $x$ splits into two daughter cells of size $\theta x$ and ($1-\theta)x$, where $\theta$ is uniformly distributed on $[\varepsilon,1-\varepsilon]$ for some $\varepsilon>0$. In this case, the infinitesimal generator of the auxiliary process is given by
\begin{align*}
\mathcal{A}_s^{(t)}f(x)=axf'(x)+2x\varphi(s)\int_{\varepsilon x}^{(1-\varepsilon)x}\left(f(y)-f(x)\right)\frac{m(y,s,t)}{m(x,s,t)}\frac{dy}{(1-2\varepsilon)x},
\end{align*}
for all $f:\mathbb{R}_+\rightarrow\mathbb{R}$ continuously differentiable, $s,t\in\mathbb{R}_+$ such that $s<t$ and $x\in \mathbb{R}_+$.

Spectral techniques fall apart in this case because of the time dependence of the division rate whereas our method works. We prove the law of large for the distribution of ancestral trajectories in this special case. In particular, we exhibit a Lyapunov function, i.e. a function $V$ satisfying the first condition of Assumption \ref{assu:ergo} below, for the time-inhomogeneous auxiliary process associated with this population dynamic and we establish the minorization condition \ref{assu:ergo}.2 detailed in Section \ref{sec:asymp}.
\paragraph*{Outline}
In Section \ref{sec:def}, we detail the structured branching process and the assumptions considered for its existence and uniqueness. Then, in Section \ref{sec:asymp}, we study the asymptotic behavior of the empirical measure: first, in Section \ref{sub:ergo}, we give our result on the ergodicity of the auxiliary process, then, in Section \ref{sub:lln}, we state the law of large numbers for the empirical distribution of ancestral trajectories for the structured branching process. Section \ref{sub:proofs} is dedicated to proofs. Finally, in Section \ref{sec:ex_inh}, we apply the techniques developed in the previous sections to study the asymptotic behavior of a size-structured population in a fluctuating environment. 
\paragraph*{Notation.}We use the classical Ulam-Harris-Neveu notation to identify each individual. Let
\[\mathcal{U}=\bigcup_{n\in\mathbb{N}}\left\lbrace 0,1\right\rbrace^{n}.
\]
The first individual is labeled by $\emptyset $. When an individual $u\in\mathcal{U}$ dies, its descendants are labeled by $u0,u1$. If $u$ is an ancestor of $v$, we write $u\leq v$.  With a slight abuse of notation, for all $u\in V_t$ and $s<t$, we denote by $X_s^u$ the trait of the unique ancestor living at time $s$ of $u$.

We also introduce the following notation for the time-inhomogeneous auxiliary process: for all measurable functions $f$, we set
\begin{align*}
\mathbb{E}_x\left(f\left(Y_s^{(t)}\right)\right):=\mathbb{E}\left(f\left(Y_s^{(t)}\right)\big|Y_0^{(t)}=x\right),
\end{align*}
for all $x\in\mathcal{X}$, $0\leq s\leq t$. 

Finally, we recall that for all $t\geq 0$ and all $0\leq r \leq s\leq t$, $P_{r,s}^{(t)}$ is also a linear operator from the set of measures of finite mass into itself through the left action. In particular, for any $x\in \mathcal{X}$, we will denote the measure $\delta_x P_{r,s}^{(t)}(dy)$ by $P_{r,s}^{(t)}(x,dy)$.

\section{The structured branching process\label{existence}}\label{sec:def}
First, we introduce some useful notations and objects to characterize the branching process. Henceforth, we work on a probability space denoted by $\left(\Omega,\mathcal{F},\mathbb{P}\right)$.

\paragraph*{Dynamic of the trait.} Let  $\mathcal{X}\subset\mathbb{R}^d$ be a measurable complete space for some $d\geq 1$. It is the state space of the Markov process describing the trait of the individuals. Let $\mathcal{G}:\mathcal{D}(\mathcal{G})\subset\mathcal{C}_b(\mathcal{X})\rightarrow \mathcal{C}_b(\mathcal{X})$ be the infinitesimal generator associated with a strongly continuous contraction semigroup where $\mathcal{C}_b(\mathcal{X})$ denotes the continuous bounded functions on $\mathcal{X}$. Then, $(X_t,t\geq 0)$ is the unique $\mathcal{X}$-valued c\`adl\`ag strong Markov process solution of the martingale problem associated with $\left(\mathcal{G},\mathcal{D}(\mathcal{G})\right)$ \cite[Theorems 4.4.1 and 4.4.2]{ethier2009markov}. We denote by $(X_t^x,t\geq 0)$ the corresponding process starting from $x\in\mathcal{X}$.

\paragraph*{Division events.} An individual with trait $x$ at time $t$ dies at an instantaneous rate $B(t,x)$, where $B$ is a continuous function from $\mathbb{R}_+\times \mathcal{X}$ to $\mathbb{R}_+$. It is replaced by two children. Their traits at birth are distributed according to the probability measure $Q(x,\cdot)$ on $\mathcal{X}^2$. We suppose that the probability measures corresponding to the marginal distributions are equal. By a slight abuse of notation, we will also denote them by $Q$.

We refer the reader to Remark \ref{rem:simplification} in the introduction for comments on the choice of model. In order to ensure the non-explosion in finite time of such a process, we need to consider the following hypotheses.
\begin{assumptionA}\label{assu:debut}
We suppose that
\begin{enumerate}
\item there exist $b_1,b_2:\mathbb{R}_+\rightarrow \mathbb{R}_+^*$ continuous and $\gamma\geq 1$ such that for all $(t,x)\in\mathbb{R}_+\times\mathcal{X}$,
\[B(t,x)\leq b_1(t)\left| x\right|^{\gamma}+b_2(t),\]
\item for all $x\in\mathcal{X}$,
\[
Y_1(x)+Y_2(x)\leq x,
\]
where the law of the couple of random variables $(Y_1(x),Y_2(x))$ is given by $Q(x,dy_1,dy_2)$,
\item for all $x\in\mathcal{X}$,
\begin{align*}
\lim_{t\rightarrow+\infty}\int_0^t B(s,X^x_s)ds=+\infty,\text{ almost surely,}
\end{align*}
\item there exists a sequence of functions $(h_{n,\gamma})_{n\in\mathbb{N}}$ such that for all $n\in\mathbb{N},\ h_{n,\gamma}\in \mathcal{D}(\mathcal{G})$ and $\lim_{n\rightarrow+\infty} h_{n,\gamma}(x)=|x|^{\gamma}$ for all $x\in\mathcal{X}$ and there exist $c_1,c_2\geq 0$ such that for all $x\in\mathcal{X}$:
\begin{align*}
\lim_{n\rightarrow+\infty}\mathcal{G}h_{n,\gamma}(x)\leq c_1|x|^{\gamma}+c_2,
\end{align*}
where $\gamma$ is defined in the first item and for $x\in\mathbb{R}^d,|x|^{\gamma}=\left(\sum_{i=1}^d |x_i|\right)^{\gamma}$.
\end{enumerate}
\end{assumptionA}
\begin{rem}
We have slightly modified the first condition on the division rate compared to the one in \cite{marguet2016} to better fit the framework of this paper. The adaptation of the proof of the non-explosion of the population to use this modified assumption is straightforward.
\end{rem}
Under Assumption \ref{assu:debut}, we have the strong existence and uniqueness of the structured branching process $Z$ in the state of c\`adl\`ag measure-valued processes, where for all $t\geq 0$,
 \[
Z_t=\sum_{u\in V_t} \delta_{X_t^u},\ t\geq 0.
\]
We refer to Theorem 2.3. in \cite{marguet2016} for more details and to \cite{fournier2004microscopic} for the study of c\`adl\`ag measure-valued processes. 

For the existence of the auxiliary process $Y^{(t)}$ with infinitesimal generators given by \eqref{eq:gene_auxi}, we need to consider additional assumptions on the mean number of individuals in the population at a given time. 
Let us define the domain of the infinitesimal generator of the auxiliary process by
\begin{align*}
\mathcal{D}(\mathcal{A})=\left\lbrace f\in\mathcal{D}(\mathcal{G})\text{ s.t. } m(\cdot,s,t)f(s,x)\in\mathcal{D}(\mathcal{G})\ \text{for all } t\geq 0\text{ and }s \leq t\right\rbrace.
\end{align*}

\begin{assumptionA}\label{assu:regularite}
We suppose that for all $t\geq 0$:
\begin{itemize}
\item[-] for all $x\in\mathcal{X}$, $s\mapsto m(x,s,t)$ is continuously differentiable on $[0,t]$,
\item[-] for all $x\in\mathcal{X}$, $f\in\mathcal{D}(\mathcal{A})$, $s\mapsto \mathcal{G}(m(\cdot,s,t)f)(x)$ is continuous.
\item[-] $\mathcal{D}(\mathcal{A})$ is dense in $\mathcal{C}_b(\mathcal{X})$ for the topology of uniform convergence.
\end{itemize}
\end{assumptionA}
This assumption allows us to derive the expression of the generator of the auxiliary process \cite[Lemma 3.4 ]{marguet2016}). It is in particular satisfied in the example developed in Section \ref{sec:ex_inh} and in the examples of \cite{marguet2016}. 
\begin{assumptionA}\label{assu:doeblin}
For all $t\geq 0$,
\[
\sup_{x\in\mathcal{X}}\sup_{s\leq t}\int_\mathcal{X}\frac{m(y,s,t)}{m(x,s,t)}Q(x,dy)<+\infty,
\]
\end{assumptionA}
This assumption tells us that we control uniformly in $x$ the benefit or the penalty of a division. In the general case, the control of the ratio $m(y,s,t)(m(x,s,t))^{-1}$ seems difficult to obtain. We refer to \cite{marguet2016} or to Section \ref{sec:ex_inh} for examples where this assumption is satisfied.

\section{Asymptotic behaviour of the structured branching process}\label{sec:asymp}
The purpose of this section is to prove the law of large numbers result. We show that asymptotically, the behavior of the whole population corresponds to the mean behavior of the auxiliary process introduced in \cite{marguet2016}.  The ergodicity of this process is the key for the proof of the law of large numbers. We notice that the ergodicity of the auxiliary process is also required for the proof of the convergence of the empirical measure in \cite{guyon2007limit}, \cite{bansaye2011limit} and \cite{cloez}. 

In Subsection \ref{sub:ergo}, we prove the ergodicity of the auxiliary process. Then, in Subsection \ref{sub:lln}, we state the main theorem of this article which is the convergence in $\mathbb{L}_2$-norm of the difference between the empirical measure and the mean value of the auxiliary process towards zero as time goes to infinity. Subsection \ref{sub:proofs} is devoted to proofs.

\subsection{Ergodicity of the auxiliary process}\label{sub:ergo}
For all $t\geq 0$, we recall that $\left(P_{r,s}^{(t)}, r\leq s\leq t\right)$ denotes the semigroup of the auxiliary process defined in \eqref{eq:gene_auxi} by its infinitesimal generators.


The next assumption gathers two classical hypotheses to obtain the ergodicity of a process \cite[Chapters 15, 16]{Meyn}. We adapt them to the time-inhomogeneous case.
\begin{assumptionA}\label{assu:ergo}
We suppose that:
\begin{enumerate}
\item there exists a function $V:\mathcal{X}\rightarrow \mathbb{R}_+$ and $c,d>0$ such that for all $x\in\mathcal{X}$, $t\geq 0$ and $s\leq t$,
\[
\mathcal{A}_s^{(t)}V(x)\leq -cV(x)+d,
\] 
\item for all $0<r<s$, there exists $\alpha_{s-r}\in (0,1)$ and a probability measure $\nu_{r,s}$ on $\mathcal{X}$ such that for all $t\geq s$,
\[
\inf_{x\in B(R,V)} P_{r,s}^{(t)}(x,\cdot)\geq \alpha_{s-r}\nu_{r,s}(\cdot),
\] 
with $B(R,V)=\left\lbrace x\in\mathcal{X}:V(x)\leq R\right\rbrace$ for some $R>\frac{2d}{c}$ where $c,d$ are defined in the first point.
\end{enumerate}	 
\end{assumptionA}  
In what follows, as in \cite{Meyn,hairer2011yet} we call {\emph{Lyapunov function}} any function $V$ satisfying the first condition of Assumption \ref{assu:ergo} and we will refer to the second point of Assumption \ref{assu:ergo} as a \emph{minorization condition}. 
Adapting directly Theorem 3.1 of \cite{hairer2011yet}, we prove that the semigroup of the auxiliary process is a contraction operator for a well-chosen norm. For all $\beta>0$, we define the following metric on $\mathcal{X}$:
\begin{align*}
d_{\beta}(x,y)=\begin{cases} 0 & x=y,\\
2+\beta V(x)+\beta V(y) & x\neq y.
\end{cases}
\end{align*}
We can now state the result on the ergodic behavior of the trajectories of auxiliary process.
\begin{prop}\label{prop:ergo}
Let $T>0$. Under Assumptions \ref{assu:debut},\ref{assu:regularite},\ref{assu:doeblin},\ref{assu:ergo}, there exists $\overline{c}>0$ and $\beta>0$ such that for all $x,y\in\mathcal{X}$, for all bounded measurable functions $F:\mathbb{D}([0,T],\mathcal{X})\rightarrow \mathbb{R}$ and all $0\leq r\leq t$, we have
\begin{align}\label{eq:ergo_traj}
\left|P_{r,t,T}F(x)-P_{r,t,T}F(y)\right|\leq Ce^{-\overline{c}(t-r)}\left\Vert F\right\Vert_{\infty}d_{\beta}(x,y).
\end{align}
where $C>0$ is a positive constant. 
\end{prop}

In the case of a division rate independent of time, the auxiliary process is still time-inhomogeneous but we obtain the convergence of the trajectories of the auxiliary process.
\begin{prop}\label{prop:conv_homogene}
Let $T\geq 0$. Assume that $B(t,x)\equiv B(x)$ for all $t\geq 0$ and $x\in\mathcal{X}$. Then, under Assumptions \ref{assu:debut},\ref{assu:regularite},\ref{assu:doeblin},\ref{assu:ergo}, there exists a probability measure $\Pi$ on the Borel $\sigma$-field of $\mathbb{D}\left([0,T],\mathcal{X}\right)$ endowed with the Skorokhod distance such that for all bounded measurable functions $F:\mathbb{D}\left([0,T],\mathcal{X}\right)\rightarrow \mathbb{R}$ and for all $x\in\mathcal{X}$,
\begin{align*}
\left|P_{0,t,T}F(x)- \Pi(F)\right|\leq Ce^{-\overline{c}t}\left\Vert F\right\Vert_{\infty}\left(2+2\beta V(x)+\beta\frac{d}{c}\right). 
\end{align*}
\end{prop}
This convergence is different from classical ergodicity results because $(P_{0,t}^{(t)},t\geq 0)$ is not a semigroup.

\subsection{A law of large numbers}\label{sub:lln}
Before stating the law of large numbers, we need to consider a final set of assumptions. For $x,y\in \mathcal{X}$ and $s>0$, let
\begin{equation}\label{eq:defphi}
\varphi_{s}(x,y)=\sup_{t\geq s} \frac{m(x,0,s)m(y,s,t)}{m(x,0,t)}.
\end{equation}
It quantifies the benefit, in term of number of individuals at time $t$, of "changing" the trait of the entire population at time $s$ by the trait $y$. This quantity is possibly infinite, but Assumption \ref{assu:J2phiNt} below ensures that it is finite. For all $x\in\mathcal{X}$, we define:
\begin{equation}\label{eq:c(x)}
c(x)=\underset{t\rightarrow\infty}\liminf\frac{\log(m(x,0,t))}{t},
\end{equation}
which corresponds to the growth rate of the total population. In particular, if the division rate is constant $B\equiv b$, we have that $c(x)\equiv b$ (see (2.6) and below in \cite{marguet2016}).

Using the same notations as \cite{marguet2016}, we set for all measurable functions $f:\mathcal{X}\rightarrow\mathbb{R}$ and for all $x\in\mathcal{X}$,
\begin{align}\label{eq:J2}
Jf(x)=2\int_{\mathcal{X}\times \mathcal{X}}f\left(y_0\right)f\left(y_1\right)Q(x,dy_0,dy_1).
\end{align}
It represents the average trait at birth of the descendants of an individual.

\begin{assumptionA}\label{assu:J2phiNt}
We suppose that
\begin{enumerate}
\item for all $x\in\mathcal{X}$, $c(x)>0$,
\item there exist $\alpha_1,D_1\geq 0$ such that $\alpha_1<c(x)$ for all $x\in\mathcal{X}$ and for all $t>0$,
\begin{equation*}
\E_{x}\left[B\left(t,Y_{t}^{(t)}\right)J\left((1\vee V(\cdot))\varphi_{t}\left(x,\cdot\right)\right)\left(Y_{t}^{(t)}\right)\right]\leq D_1e^{\alpha_1 t},
\end{equation*}
where $V$ is defined in Assumption \ref{assu:ergo}.
\end{enumerate}
\end{assumptionA} 
By the definition of $c(x)$, the first point ensures that the growth of the population is exponential (which is not the case, for example, if the trait of the initial individual remains constant at a value where $B$ is equal to zero). This condition is satisfied for instance if the division rate is lower bounded by a positive constant or in the example given in the last section. The second point is a technical assumption. In particular, if $\varphi_t,B,V$ are upper bounded by polynomials and if we can control the moments of the measure $m$, the first point of Assumption \ref{assu:J2phiNt} amounts to bounding the moments of the auxiliary process. We refer the reader to Lemma \ref{lemma:verif_J2Nt} in the last section of this article for the verification of this hypothesis in an example.

We first state a slightly less strong result than the law of large numbers. 
\begin{thm}\label{thm:weakl2}
Let $T>0$. Under Assumptions \ref{assu:debut},\ref{assu:regularite},\ref{assu:doeblin},\ref{assu:ergo},\ref{assu:J2phiNt}, we have for all bounded measurable functions $F:\mathbb{D}([0,T],\mathcal{X})\rightarrow\mathbb{R}$, for all   $x_0,x_1\in\mathcal{X}$,
\begin{align}\label{eq:weakl2}
\mathbb{E}_{\delta_{x_0}}\left[\left(\sum_{u\in V_{t+T}}\frac{F\left(X^u_{t+s},s\leq T\right)-P_{0,t,T}F\left(x_1\right)}{m(x_0,0,t+T)}\right)^2\right]\underset{t\rightarrow \infty}{\longrightarrow}0.
\end{align}
Moreover, the rate of convergence is lower-bounded by:
\begin{align*}
v(t)=\exp\left(\min\left(\overline{c},\frac{c(x_0)-\alpha_1}{2}\right)t\right),
\end{align*}
where $\overline{c}$ is defined below in \eqref{eq:cbar}.
\end{thm}
As in \cite{guyon2007limit} and \cite{bansaye2011limit}, we could generalize this result to unbounded functions $F$ satisfying specific conditions such as $P_{0,t}^{(t)}F\leq e^{bt}$ for some $b<c(x)$. The rate of convergence of the empirical measure depends both on the growth rate of the population and on the rate that governs the exponential ergodicity for the auxiliary process. The same type of rate of convergence appeared in \cite[Theorem 3]{HO2016}, in the case of an age structured population. 

In order to derive the law of large numbers from the previous result, we need to control the variance of the number of individuals in the population. 
\begin{assumptionA}\label{assu:Ntmt}
For all $x\in\mathcal{X}$,
\[\sup_{t\geq 0}\mathbb{E}_{\delta_x}\left(\left(\frac{N_t}{m(x,0,t)}\right)^2\right)<\infty.
\]
\end{assumptionA}
The meaning of this assumption is that the number of individuals at time $t$ in the population is of the same order as the expected number of individuals in the population at time $t$.
We can now state the law of large numbers.
\begin{cor}\label{cor:lln}
Let $T>0$. Under Assumptions \ref{assu:debut},\ref{assu:regularite},\ref{assu:doeblin},\ref{assu:ergo},\ref{assu:J2phiNt},\ref{assu:Ntmt}, for all bounded measurable functions $F:\mathbb{D}([0,T],\mathcal{X})\rightarrow\mathbb{R}$, for all $x_0,x_1\in\mathcal{X}$, we have,
\begin{align*}
\frac{\sum_{u\in V_{t+T}}F\left(X_{t+s}^{u},s\leq T\right)}{N_{t+T}}-P_{0,t,T}F(x_1)\xrightarrow[t\rightarrow +\infty]{} 0\text{, in }\mathbb{L}_2(\delta_{x_0}).
\end{align*}
\end{cor}
\begin{rem}
It is possible to extend this convergence to population processes allowing death events i.e. if $p_0\not\equiv 0$. In this case, the convergence is only valid on the survival event $\left\lbrace N_t>0\right\rbrace$.
\end{rem}
\begin{rem}
We are not able to give the rate of convergence in this case because we did not prove the convergence of $(N_tm(x,t)^{-1},t\geq 0)$, for $x\in\mathcal{X}$.
\end{rem}
In the case of a division rate that does not depend on time, even if the auxiliary process is still time-inhomogeneous, it converges when time goes to infinity according to Proposition \ref{prop:conv_homogene}. Therefore, we obtain the following result. 
\begin{cor}\label{cor:lln_hom}
Let $T>0$. Under Assumptions \ref{assu:debut},\ref{assu:regularite},\ref{assu:doeblin},\ref{assu:ergo},\ref{assu:J2phiNt},\ref{assu:Ntmt}, if $B(t,x)\equiv B(x)$ for all $t\geq 0$ and $x\in\mathcal{X}$, there exists a probability measure $\Pi$ on the Borel $\sigma$-field of $\mathbb{D}\left([0,T],\mathcal{X}\right)$ endowed with the Skorokhod distance such that:
\begin{align*}
\frac{\sum_{u\in V_{t+T}}F\left(X_{t+s}^{u},s\leq T\right)}{N_{t+T}}\xrightarrow[t\rightarrow +\infty]{} \Pi(F)\text{, in }\mathbb{L}_2(\delta_{x_0}).
\end{align*}
\end{cor}
Therefore, the empirical measure of ancestral trajectories converges toward the limit of the auxiliary process. 
\subsection{Proofs}\label{sub:proofs}
We first give a useful inequality. Combining the first point of Assumption \ref{assu:ergo} and Dynkin's formula applied to $x\mapsto e^{ct}V(x)$ where $c,V$ are defined in Assumption \ref{assu:ergo}, we have,
\begin{align}\label{eq:assu1avecsemi}
P_{r,s}^{(t)}V(x)\leq e^{-c(s-r)}V(x)+\frac{d}{c}\left(1-e^{-c(s-r)}\right).
\end{align}
We will use this inequality in the two following subsections. 
\subsubsection{Proof of Proposition \ref{prop:ergo}}
This is adapted from \cite[Theorem 3.1]{hairer2011yet}. We consider the semi-norm on measurable functions from $\mathcal{X}$ into $\mathbb{R}$ defined by
\begin{align*}
\vvvert f\vvvert_{\beta}=\sup_{x\neq y}\frac{\left|f(x)-f(y)\right|}{d_\beta (x,y)}.
\end{align*}
We also introduce the following weighted norm:
\begin{align*}
\left\Vert f\right\Vert_{\beta}=\sup_x\frac{|f(x)|}{1+\beta V(x)}.
\end{align*}
\paragraph*{Step 1.}Let $0\leq r\leq s\leq t$ and $f:\mathcal{X}\rightarrow\mathbb{R}$ be a bounded measurable function. First, we prove that for all $\Delta>0$, there exists $\overline{\alpha}_{\Delta}\in (0,1)$ and $\beta_{\Delta}>0$ such that for all $r>0$ and all $t\geq r+\Delta$,
\begin{align}\label{eq:contract}
\vvvert P_{r,r+\Delta}^{(t)}f\vvvert_{\beta_{\Delta}}\leq \overline{\alpha}_{\Delta}\vvvert f\vvvert_{\beta_{\Delta}}.
\end{align}
Let $\beta>0$ that will be specified later. Fix $R>\frac{2d}{c}$ and $f:\mathcal{X}\rightarrow \mathbb{R}$ such that $\vvvert f\vvvert_{\beta}\leq 1$. Using Lemma 2.1 in \cite{hairer2011yet}, we can assume without loss of generality that $\left\Vert f\right\Vert_{\beta}\leq 1$. To obtain \eqref{eq:contract}, it is sufficient to prove that for all $x,y\in\mathcal{X}$, there exists $\overline{\alpha}_{\Delta}\in(0,1)$ and $\beta_{\Delta}>0$ such that
\begin{align*}
\left|P_{r,r+\Delta}^{(t)}f(x)-P_{r,r+\Delta}^{(t)}f(y)\right|\leq \overline{\alpha}_{\Delta}d_{\beta_{\Delta}}(x,y).
\end{align*}
If $x=y$, the claim is true. Let $x\neq y\in\mathcal{X}$.
We assume first that $x$ and $y$ are such that
\begin{align*}
V(x)+V(y)\geq R.
\end{align*} 
Then, we have
\begin{align*}
\left|P_{r,r+\Delta}^{(t)}f(x)-P_{r,r+\Delta}^{(t)}f(y)\right|\leq 2+\beta P_{r,r+\Delta}^{(t)}V(x)+\beta P_{r,r+\Delta}^{(t)}V(y),
\end{align*}
because $\left\Vert f\right\Vert_{\beta}\leq 1$. Next, using \eqref{eq:assu1avecsemi}, we obtain
\begin{align*}
\left|P_{r,r+\Delta}^{(t)}f(x)-P_{r,r+\Delta}^{(t)}f(y)\right|&\leq 2 +\beta e^{-c\Delta}\left(V(x)+V(y)\right)+2\beta\frac{d}{c}\left(1-e^{-c\Delta}\right)\\
&\leq 2 +\beta e^{-c\Delta}\left(V(x)+V(y)\right)+2\beta\frac{d}{Rc}(V(x)+V(y))\left(1-e^{-c\Delta}\right).
\end{align*}
Let $\gamma_{\Delta}^0=e^{-c\Delta}+\frac{2d}{Rc}\left(1-e^{-c\Delta}\right)$. We have $\gamma_{\Delta}^0<1$. Then,
\begin{align}\label{eq:E1}\nonumber
\left|P_{r,r+\Delta}^{(t)}f(x)-P_{r,r+\Delta}^{(t)}f(y)\right|&\leq 2+\beta \gamma_{\Delta}^0(V(x)+V(y))\\\nonumber
&\leq \left(\frac{2+\gamma_{\Delta}^0\beta(V(x)+V(y))}{2+\beta (V(x)+ V(y))}\right)(2+\beta V(x)+\beta V(y))\\
&\leq \gamma_{\Delta}^1d_{\beta}(x,y),
\end{align}
where
\[
\gamma_{\Delta}^1=\frac{2+\beta R\gamma_{\Delta}^0}{2+\beta R}<1.
\]
Assume now that $x$ and $y$ are such that
\[
V(x)+V(y)< R.
\]
Let us consider the following linear operator: $$\widetilde{P}_{r,r+\Delta}^{(t)}=\frac{1}{1-\alpha_{\Delta}}P_{r,r+\Delta}^{(t)}-\frac{\alpha_{\Delta}}{1-\alpha_{\Delta}}\nu_{r,r+\Delta},$$
where $\alpha_{\Delta}$ is given in Assumption \ref{assu:ergo}2.
We have
\begin{align*}
\left|P_{r,r+\Delta}^{(t)}f(x)-P_{r,r+\Delta}^{(t)}f(y)\right|=(1-\alpha_{\Delta})|\widetilde{P}_{r,r+\Delta}^{(t)}f(x)-\widetilde{P}_{r,r+\Delta}^{(t)}f(y)|.
\end{align*}
According to the second point of Assumption \ref{assu:ergo}, $\widetilde{P}_{r,r+\Delta}^{(t)}f(x)\geq 0$ for all $f\geq 0$ and $x\in B(R,V)$. Then,
\begin{align*}
\left|P_{r,r+\Delta}^{(t)}f(x)-P_{r,r+\Delta}^{(t)}f(y)\right|\leq (1-\alpha_{\Delta})\left(\widetilde{P}_{r,r+\Delta}^{(t)}f(x)+\widetilde{P}_{r,r+\Delta}^{(t)}f(y)\right).
\end{align*}
Next, using that $\left\Vert f\right\Vert _{\beta}\leq 1$ and that $\widetilde{P}_{r,r+\Delta}^{(t)}V(x)\leq \frac{1}{1-\alpha_{\Delta}}P_{r,r+\Delta}^{(t)}V(x)$, we get
\begin{align*}
\left|P_{r,r+\Delta}^{(t)}f(x)-P_{r,r+\Delta}^{(t)}f(y)\right|&\leq 2\left(1-\alpha_{\Delta}\right)+\beta\left(P_{r,r+\Delta}^{(t)}V(x)+P_{r,r+\Delta}^{(t)}V(y)\right)\\
&\leq 2\left(1-\alpha_{\Delta}+\beta\frac{d}{c}\left(1-e^{-c\Delta}\right)\right)+\beta e^{-c\Delta}(V(x)+V(y)),
\end{align*}
where the second inequality comes from \eqref{eq:assu1avecsemi}. 
Let $\alpha_{\Delta}^0\in(0,\frac{2d}{Rc}\alpha_{\Delta})$. Then, fixing $$\beta=\beta_{\Delta}:=cd^{-1}\alpha_{\Delta}^0(1-e^{-c\Delta})^{-1},$$ yields
\begin{align}\label{eq:E2}\nonumber
\left|P_{r,r+\Delta}^{(t)}f(x)-P_{r,r+\Delta}^{(t)}f(y)\right|&\leq 2\left(1-\alpha_{\Delta}+\alpha_{\Delta}^0\right)+\beta_{\Delta} e^{-c\Delta}(V(x)+V(y))\\
&\leq \gamma_{\Delta}^2 d_{\beta_{\Delta}}(x,y),
\end{align}
where $$\gamma_{\Delta}^2=e^{-c\Delta}\vee (1-(\alpha_{\Delta}-\alpha_{\Delta}^0)).$$
Finally, combining \eqref{eq:E1} and \eqref{eq:E2} and noticing that $\gamma_{\Delta}^1>e^{-c\Delta}$, yields the result with $\overline{\alpha}_{\Delta}=\gamma_{\Delta}^1\vee (1-(\alpha_{\Delta}-\alpha_{\Delta}^0))$.

\paragraph*{Step 2.} We now prove \eqref{eq:ergo_traj}. Conditioning with respect to $\sigma\left(Y_{u}^{(t+T)},r\leq u\leq t\right)$ and using the Markov property, we obtain
\begin{align}\label{eq:1}
P_{r,t,T}F(x)-P_{r,t,T}F(y)&=\int_{\mathcal{X}}P_{t,t,T}F(z)\left(P_{r,t}^{(t+T)}(x,dz)-P_{r,t}^{(t+T)}(y,dz)\right).
\end{align}
For all $z\in\mathcal{X}$, we set $g(z)=P_{t,t,T}F(z)$. Let $\Delta>0$. Let $l(r,t)\in\mathbb{N}$ and $ \varepsilon_{r,t}\geq 0$ be such that $t-r=l(r,t)\Delta+\varepsilon_{r,t}$ and $\varepsilon_{r,t}<\Delta$. Using \eqref{eq:contract}, we have
\begin{align*}
\left|P_{r,t}^{(t+T)}g(x)-P_{r,t}^{(t+T)}g(y)\right|&=\left|P_{r,r+\Delta}^{(t+T)}P_{r+\Delta,t}^{(t+T)}g(x)-P_{r,r+\Delta}^{(t+T)}P_{r+\Delta,t}^{(t+T)}g(y)\right|\\
&\leq \overline{\alpha}_{\Delta}d_{\beta_{\Delta}}(x,y)\vvvert P_{r+\Delta,t}^{(t+T)}g\vvvert_{\beta_{\Delta}}\\
&\leq \left(\overline{\alpha}_{\Delta}\right)^{l(r,t)}d_{\beta}(x,y)\left\Vert g\right\Vert_{\infty},
\end{align*}
where $\beta = cd^{-1}$. Finally, we obtain
\begin{align}\label{eq:2}
\left|P_{r,t}^{(t+T)}g(x)-P_{r,t}^{(t+T)}g(y)\right|\leq Ce^{-\overline{c}(t-r)}d_{\beta}(x,y)\left\Vert g\right\Vert_{\infty},
\end{align}
where $C:=1+\frac{cR}{2d}$ and
\begin{align}\label{eq:cbar}
\overline{c}:=\sup_{\Delta>0}\log(\overline{\alpha}_{\Delta}^{-1})\Delta^{-1}.
\end{align}
In particular, $\overline{c}<c$ because $\overline{\alpha}_{\Delta}>e^{-c\Delta}$. Finally, combining \eqref{eq:1} and \eqref{eq:2}, and using that $\left\Vert P_{t,t+T}^{(t+T)}F\right\Vert_{\infty}\leq \left\Vert F\right\Vert_{\infty}$,  we get the result. 
\subsubsection{Proof of Proposition \ref{prop:conv_homogene}}
Let $F:\mathbb{D}\left([0,T],\mathcal{X}\right)\rightarrow \mathbb{R}$ be a bounded measurable function. We have for all $t,r\geq 0$,
\begin{align*}
P_{0,t+r,T}F(x)=\mathbb{E}_x\left[F\left(Y_{t+r+s}^{(t+r+T)},s\leq T\right)\right]=\mathbb{E}_x\left[\mathbb{E}\left[F\left(Y_{t+r+s}^{(t+r+T)},s\leq T\right)\big|\mathcal{F}_{r}^{(t+r+T)}\right]\right].
\end{align*}
Using the Markov property, we have
\begin{align*}
P_{0,t+r,T}F(x)=\int_{\mathcal{X}}P_{r,t+r,T}F(y)P_{0,r}^{(t+r+T)}(x,dy).
\end{align*}
Since $B$ does not depend on time, we have $m(y,r,t+r+T)=m(y,0,t+T)$. Then, using the Many-to-One formula \eqref{mtomesurable} and the Markov property, we get
\begin{align*}
P_{r,t+r,T}F(y)=\frac{\mathbb{E}\left[\sum_{u\in V_{t+T}}F\left(X_{t+s}^{u},s\leq T\right)\big|Z_{0}=\delta_y\right]}{m(y,0,t+T)}=P_{0,t,T}F(y),
\end{align*}
so that
\begin{align*}
P_{0,t+r,T}F(x)=\int_{\mathcal{X}}P_{0,t,T}F(y)P_{0,r}^{(t+r+T)}(x,dy).
\end{align*}
Next,
\begin{align*}
\left|P_{0,t+r,T}F(x)-P_{0,t,T}F(x)\right|\leq \int_{\mathcal{X}}\left|P_{0,t,T}F(y)-P_{0,t,T}F(x)\right|P_{0,r}^{(t+r+T)}(x,dy).
\end{align*}
Then, according to \eqref{eq:ergo_traj}, there exist $\overline{c}>0$, $\beta>0$ and a constant $C>0$ such that
\begin{align*}
\left|P_{0,t+r,T}F(x)-P_{0,t,T}F(x)\right|&\leq Ce^{-\overline{c}t}\left\Vert F \right\Vert_{\infty}\int_{\mathcal{X}}\left(2+\beta V(y)+\beta V(x)\right)P_{0,r}^{(t+r+T)}(x,dy)\\
&\leq Ce^{-\overline{c}t}\left\Vert F \right\Vert_{\infty}\left(2+2\beta V(x)+\beta\frac{d}{c}\right)\xrightarrow[r,t\rightarrow +\infty]{} 0,
\end{align*}
where the last inequality comes from \eqref{eq:assu1avecsemi}. 
Finally, $(P_{0,t,T}F(x),t\geq 0)$ is a Cauchy sequence in $\mathcal{X}$ which is complete. Then, it has a limit as $t\rightarrow +\infty$ and this limit is independent of $x$ by \eqref{eq:ergo_traj}.
\subsubsection{Proof of Theorem \ref{thm:weakl2}, Corollary \ref{cor:lln} and Corollary \ref{cor:lln_hom}}\label{sec:proof_lln} 
Let $F:\mathbb{D}([0,T],\mathcal{X})\rightarrow \mathbb{R}_+$ be a bounded measurable function. 
For all $x\in\mathbb{D}([0,t+T],\mathcal{X})$ and $x_1\in\mathcal{X}$, we define the following function:\[\phi_{t,T}(x_1,(x_s,s\leq t+T))=F\left(x_{t+s},s\leq T\right)-P_{0,t,T}F(x_1).\] 
\begin{proof}[Proof of Theorem \ref{thm:weakl2}]
Fix $x_0\in\mathcal{X}$. Let $\varepsilon>0$ be such that $c(x_0)-\alpha_1>\varepsilon$, where $\alpha_1$ is defined in Assumption \ref{assu:J2phiNt}. Let $t>0$ be such that $$c(x_0)<\inf_{s\geq t}\frac{\log(m(x_0,0,s))}{s}+\varepsilon.$$ 
We have
\begin{align*}
&\mathbb{E}_{\delta_{x_0}}\left[\left(\sum_{u\in V_{t+T}}\frac{\phi_{t,T}(x_1,(X^u_s,s\leq t+T))}{m(x_0,0,t+T)}\right)^2\right]=A(t,T)+B(t,T),
\end{align*}
where
\begin{align*}
&A(t,T)=m(x_0,0,t+T)^{-2}\mathbb{E}_{\delta_{x_0}}\left[\sum_{u\in V_{t+T}}\phi_{t,T}(x_1,\left(X^u_s, s\leq t+T\right))^2\right],\\
&B(t,T)=m(x_0,0,t+T)^{-2}\mathbb{E}_{\delta_{x_0}}\left[\sum_{u\neq v\in V_{t+T}}\phi_{t,T}(x_1,\left(X^u_s, s\leq t+T\right))\phi_{t,T}(x_1,\left(X^v_s, s\leq t+T\right))\right].
\end{align*}
For the first term, using that $\phi_{t,T}(x_1,\left(X^u_s, s\leq t+T\right))^2\leq 4\left\Vert F\right\Vert_{\infty}^2$, we get
\begin{align*}
A(t,T)\leq 4 e^{-(c(x_0)-\varepsilon)(t+T)}\left\Vert F\right\Vert_{\infty}^2\xrightarrow[t\rightarrow +\infty]{} 0.
\end{align*}
For the second term, using the Many-to-One formula for forks \cite[Proposition 3.6]{marguet2016}, we have
\begin{align*}
m(x_0,0,t+T)^2B(t,T)=\int_0^{t+T} m(x_0,0,s)\mathbb{E}_{x_0}\left[B\left(Y_s^{(s)}\right)J_{s,t+T}\phi_{t,T}(x_1,\cdot)\left(Y_{r}^{(s)},r\leq s\right)\right]ds,
\end{align*}
where for $x\in\mathbb{D}([0,s],\mathcal{X})$,
\begin{align*}
J_{s,t+T}\phi_{t,T}(x_1,\cdot)\left(x\right)=&2\int_{\mathcal{X}^2} m\left(y_0,s,t+T\right)\mathbb{E}\left[\phi_{t,T}\left(x_1,\left(\widetilde{Y}^{(t+T)}_{r},r\leq t+T\right)\right)\middle|Y_s^{(t+T)}=y_0\right]\\
& m\left(y_1,s,t+T\right)\mathbb{E}\left[\phi_{t,T}\left(x_1,\left(\widetilde{Y}^{(t+T)}_{r},r\leq t+T\right)\right)\middle|Y_s^{(t+T)}=y_1\right]Q(x_s,dy_0,dy_1),
\end{align*}
where
\begin{align*}
\widetilde{Y}^{(t+T)}_{r}=\left\{\begin{array}{lc}
x_r&\text{ if }r<s,\\
Y_{r}^{(t+T)}&\text{ if }s\leq r\leq t+T.
\end{array}\right.
\end{align*}
We split the integral into two parts:
\begin{align*}
B(t,T)=I_1+I_2,
\end{align*}
where
\begin{align*}
I_1&=m(x_0,0,t+T)^{-2}\int_{t}^{t+T} m(x_0,0,s)\mathbb{E}_{x_0}\left[B\left(Y_s^{(s)}\right)J_{s,t+T}\phi_{t,T}(x_1,\cdot)\left(Y_{r}^{(s)},r\leq s\right)\right]ds,\\
I_2&=m(x_0,0,t+T)^{-2}\int_{0}^{t} m(x_0,0,s)\mathbb{E}_{x_0}\left[B\left(Y_s^{(s)}\right)J_{s,t+T}\phi_{t,T}(x_1,\cdot)\left(Y_{r}^{(s)},r\leq s\right)\right]ds.
\end{align*}
For the first integral, we have
\begin{align*}
I_1&\leq 4\left\Vert F\right\Vert_{\infty}^2 \int_{t}^{t+T} m(x_0,0,s)^{-1}\mathbb{E}_{x_0}\left[B\left(Y_s^{(s)}\right)J\varphi_s\left(x_0,\cdot\right)\left(Y_s^{(s)}\right)\right]ds\\
&\leq 4\left\Vert F\right\Vert_{\infty}^2 \int_{t}^{t+T} e^{-(c(x_0)-\varepsilon)s}D_1e^{\alpha_1 s}ds\\
&\leq 4\left\Vert F\right\Vert_{\infty}^2\frac{D_1}{c(x_0)-\alpha_1-\varepsilon}e^{\left(\alpha_1-c(x_0)+\varepsilon\right)t}\xrightarrow[t\rightarrow +\infty]{} 0,
\end{align*}
where the second inequality comes from Assumption \ref{assu:J2phiNt}.
Therefore, we only have to deal with the remaining integral $I_2$.
First, we notice that for any $0\leq s\leq t$ and $0\leq r\leq T$,
\begin{align*}
\widetilde{Y}^{(t+T)}_{t+r}=Y_{t+r}^{(t+T)}.
\end{align*}
Therefore, we get
\begin{align*}
\phi_{t,T}\left(x_1,\left(\widetilde{Y}^{(t+T)}_{r},r\leq t+T\right)\right)=\phi_{t,T}\left(x_1,\left(Y^{(t+T)}_{r},r\leq t+T\right)\right).
\end{align*}
Next, Assumption \ref{assu:J2phiNt} yields
\begin{align*}
I_2\leq \int_{0}^{t}&m(x_0,0,s)^{-1}\\
&\times\mathbb{E}_{x_0}\left[B\left(Y_s^{(s)}\right)J\left(\varphi_s(x_0,\cdot)\E\left(\phi_{t,T}\left(x_1,\left(Y^{(t+T)}_{r},r\leq t+T\right)\right)\middle|Y_s^{(t+T)}=\cdot\right)\right)\left(Y_{s}^{(s)}\right)\right]ds.
\end{align*}
Moreover, for any $y\in\mathcal{X}$ and $s\leq t$, we have
\begin{align*}
\E\left(\phi_{t,T}\left(x_1,\left(Y^{(t+T)}_{r},r\leq t+T\right)\right)\Big|Y_{s}^{(t+T)}=y\right)=P_{s,t,T}F(y)-P_{0,t,T}F(x_1).
\end{align*}
According to Proposition \ref{prop:ergo}, there exists $\overline{c}>0$, $\beta>0$ and $C>0$ such that
\begin{align*}
\left|P_{s,t,T}F(y)-P_{0,t,T}F(x_1)\right|\leq Ce^{-\overline{c}(t-s)}\left\Vert P_{t,t+T}^{(t+T)}F\right\Vert_{\infty}\int_{\mathcal{X}}	d_{\beta}(y,x_2)P_{0,s}^{(t+T)}(x_1,dx_2).
\end{align*}
Finally:
\begin{align*}
\left|P_{s,t,T}F(y)-P_{0,t,T}F(x_1)\right|\leq Ce^{-\overline{c}(t-s)}\left\Vert F\right\Vert_{\infty}\left(2+\beta V(y)+\beta V(x_1)+\beta \frac{d}{c}\right).
\end{align*}
Then, we have
\begin{align*}
I_2\leq& C\left\Vert F\right\Vert_{\infty}^2\int_0^{t}e^{-2\overline{c}(t-s)}m(x_0,0,s)^{-1}\\
&\times\mathbb{E}_{x_0}\left[B\left(Y_s^{(s)}\right)J\left(\varphi_s(x_0,\cdot)\left(2+\beta V(\cdot)+\beta V(x_1)+\beta \frac{d}{c}\right)\right)\left(Y_{s}^{(s)}\right)\right]ds.
\end{align*}
Next, using Assumption \ref{assu:J2phiNt} we obtain
\begin{align*}
I_2\leq& C \left\Vert F\right\Vert_{\infty}^2\left(2+\beta+\beta V(x_1)+\beta\frac{d}{c}\right)\int_0^{t}e^{-2\overline{c}(t-s)}e^{(\alpha_1-c(x_0)+\varepsilon)s}ds,
\end{align*}
where $C>0$ denotes a positive constant which can vary from line to line. Then,
\begin{align*}
I_2\leq C &\left\Vert F\right\Vert_{\infty}^2e^{-2\overline{c}t}\int_0^{t}e^{(\alpha_1-c(x_0)+2\overline{c}+\varepsilon)s}ds\\
\leq C &\frac{\left\Vert F\right\Vert_{\infty}^2}{\alpha_1-c(x_0)+2\overline{c}+\varepsilon}e^{-2\overline{c}t}\left(e^{(\alpha_1-c(x_0)+2\overline{c}+\varepsilon)t}-1\right)\\
\leq C &\frac{\left\Vert F\right\Vert_{\infty}^2}{\alpha_1-c(x_0)+2\overline{c}+\varepsilon}\left(e^{(\alpha_1-c(x_0)+\varepsilon)t}-e^{-2\overline{c}t}\right)\\
\leq C &\left\Vert F\right\Vert_{\infty}^2 e^{-\min(2\overline{c},c(x_0)-\alpha_1-\varepsilon)t}. 
\end{align*}
Finally, we obtain
\begin{align*}
A(t,T)+B(t,T) \leq C \left\Vert F\right\Vert_{\infty}^2 e^{-\min(2\overline{c},c(x_0)-\alpha_1-\varepsilon)t}
\end{align*}
where $C$ is a constant depending on $x_0, \beta, V(x_1), c, d,c(x_0),\alpha_1,R$.
\end{proof}
We now prove Corollary \ref{cor:lln}.
\begin{proof}[Proof of Corollary \ref{cor:lln}]
Let $T>0$, $\varepsilon>0$, $x_0\in\mathcal{X}$ and let $F:\mathbb{D}([0,T],\mathcal{X})\rightarrow\mathbb{R}$ be a bounded measurable function. Let $\delta>0$. We have
\begin{align*}
\mathbb{E}_{\delta_{x_0}}\left[\left(\frac{\sum_{u\in V_{t+T}}\phi_{t,T}\left(x_1,\left(X_s^{u},s\leq t+T\right)\right)}{N_{t+T}}\right)^2\right]\leq & \delta^2\mathbb{E}_{\delta_{x_0}}\left[\left(\frac{\sum_{u\in V_{t+T}}\phi_{t,T}\left(x_1,\left(X_s^{u},s\leq t+T\right)\right)}{m(x_0,0,t+T)}\right)^2\right]\\
&+4\left\Vert F\right\Vert_{\infty}^2\mathbb{P}_{\delta_{x_0}}\left(N_tm(x_0,0,t+T)^{-1}\leq \delta^{-1}\right).
\end{align*}
According to Paley-Zygmund inequality and Assumption \ref{assu:Ntmt}, we have
\begin{align}\label{eq:preuve_lln}
\mathbb{P}_{\delta_{x_0}}\left(N_t\leq \delta^{-1} m(x_0,0,t+T)\right)\leq 1-(1-\delta^{-1})^2\mathbb{E}_{\delta_{x_0}}\left[\left(\frac{N_{t+T}}{m(x_0,0,t+T)}\right)^2\right]^{-1}\leq 1- \frac{(1-\delta^{-1})^2}{g(x_0)},
\end{align}
where $g:\mathcal{X}\rightarrow \mathbb{R}_+$ is such that for all $x_0\in\mathcal{X}$, we have $\mathbb{E}_{\delta_{x_0}}\left[N_{t+T}^2m(x_0,0,t+T)^{-2}\right]\leq g(x_0).$
Finally, we can fix $\delta$ such that, combining \eqref{eq:preuve_lln} and Theorem \ref{thm:weakl2}, for $t$ large enough, we have
\begin{align*}
\mathbb{E}_{\delta_{x_0}}\left[\left(\frac{\sum_{u\in V_{t+T}}\phi_{t,T}\left(x_1,\left(X_s^{u},s\leq t+T\right)\right)}{N_{t+T}}\right)^2\right]\leq \varepsilon.
\end{align*}
\end{proof}

Corollary \ref{cor:lln_hom} is a direct consequence of Proposition \ref{prop:conv_homogene} and Corollary \ref{cor:lln}. 
\section{Asymptotic behavior a time-inhomogeneous dynamic: application of ergodicity techniques}\label{sec:ex_inh}
In the study of population dynamics, time-inhomogeneity typically appears in fluctuating environment. This effect can be modeled by a division rate that changes over time. In this section, we show how our method via the ergodicity of the auxiliary process applies to such models.

We consider a size-structured cell population in a fluctuating environment: each cell grows exponentially at rate $a>0$ and division occurs at time $t$ at rate $B(t,x)= x\varphi(t)$, if $x$ is the size of the cell at time $t$. We assume that $\varphi:\mathbb{R}_+\rightarrow\mathbb{R}_+$ is continuous and that there exist $\varphi_1,\varphi_2>0$ such that for all $t\in\mathbb{R}_+$,
\begin{align*}
\varphi_1\leq \varphi(t)\leq \varphi_2.
\end{align*}
The choice $B(x) = x$ is classical in the study of growth-fragmentation equations \cite{mischler2016spectral}. The originality comes from the function $\varphi$ which models a changing environment. 

At division, the cell splits into two daughter cells of size $\theta x$ and $(1-\theta)x$, with $\theta\sim\mathcal{U}\left([\varepsilon,1-\varepsilon]\right)$ for some $0<\varepsilon<\frac{1}{2}$ and $x$ the size of the cell at division. Then, the process that we consider is a Piecewise Deterministic Markov Process (PDMP) on a tree with individual jump rate $B$ and transition density function $Q$ given by
\begin{align*}
Q(x,y)=\left\{\begin{array}{lc}
\frac{1}{(1-2\varepsilon)x}&\text{ if }\varepsilon x\leq y\leq (1-\varepsilon)x,\\
0 &\text{ otherwise. }
\end{array}\right.
\end{align*}

Let us first make some comments on the choice of the model. The function $\varphi$ is lower bounded to ensure that each cell effectively divides after some time. The upper bound is convenient for the calculations. An interesting example is $B(t,x)=x(\alpha+\beta\sin(t))$, with $\alpha-\beta>0$ for the modeling of the growth of a cell population in a periodic environment. Finally, we consider a uniform law on $[\varepsilon,1-\varepsilon]$ for the kernel at division but the next lemmas can easily be extend to a more general kernel.  

Following the same calculations as in \cite[Section 2.2]{marguet2016}, we have
\begin{align*}
m(x,s,t)=1+x\phi(s,t),\  \forall x\in\mathbb{R}_+.
\end{align*}
where
\begin{align*}
\phi(s,t)=\int_s^t \varphi(r)e^{a(r-s)}dr.
\end{align*} 
Moreover, in this case, the infinitesimal generator of the auxiliary process is given by
\begin{align}\label{eq:gene_ex}
\mathcal{A}_s^{(t)}f(x)=axf'(x)+2x\varphi(s)\int_{\varepsilon x}^{(1-\varepsilon)x}\left(f(y)-f(x)\right)\frac{m(y,s,t)}{m(x,s,t)}\frac{dy}{(1-2\varepsilon)x},
\end{align}
for all $f:\mathbb{R}_+\rightarrow\mathbb{R}$ continuously differentiable, all $s,t\in\mathbb{R}_+$ such that $s<t$ and all $x\in \mathbb{R}_+$.
Then, the division rate of the auxiliary process is given by
\begin{align*}
\widehat{B}_s^{(t)}(x)=2x\varphi(s)\frac{m(x/2,s,t)}{m(x,s,t)},
\end{align*}
and the transition kernel for the trait at birth is given by
\begin{align*}
\widehat{Q}_s^{(t)}(x,dy)=\widehat{Q}_s^{(t)}(x,y)dy=\frac{m(y,s,t)}{x(1-2\varepsilon)m(x/2,s,t)}\mathbf{1}_{\varepsilon x\leq y\leq (1-\varepsilon)x}dy,
\end{align*}
where $$m(x/2,s,t)=\int_{\varepsilon x}^{(1-\varepsilon)x}m(y,s,t)Q(x,y)dy,$$
corresponds to a normalization term so that $\widehat{Q}_s^{(t)}(x,dy)$ is a probability measure.
We have the following result on the asymptotic behavior of the measure-valued branching process.
\begin{thm}\label{th:llg_example}
Let $T>0$. For all bounded measurable functions $F:\mathbb{D}([0,T],\mathcal{X})\rightarrow\mathbb{R}$, for all $x_0,x_1\in\mathcal{X}$, we have
\begin{align}\label{eq:conv_lln}
\frac{\sum_{u\in V_{t+T}}F\left(X_{t+s}^{u},s\leq T\right)}{N_{t+T}}-\mathbb{E}_{x_1}\left[F\left(Y_{t+s}^{(t+T)},s\leq T\right)\right]\xrightarrow[t\rightarrow +\infty]{} 0\text{, in }\mathbb{L}_2(\delta_{x_0}).
\end{align}
\end{thm}
The proof of Theorem \ref{th:llg_example} is detailed in several lemmas. First, in Lemma \ref{lemma:Lyapunov_min}, we exhibit a Lyapunov function and a probability measure which ensure that Assumption \ref{assu:ergo} is satisfied. Next, in Lemma \ref{lemma:moment}, we prove that the moments of the auxiliary process are bounded. Finally, in Lemmas \ref{lemma:verif_J2Nt} and \ref{lemma:varWt}, we prove that Assumptions \ref{assu:J2phiNt} and \ref{assu:Ntmt} are satisfied. 

Let $V(x)=\frac{1}{x}+x$ for $x\in\mathbb{R}_+^{*}$. We recall that $B(R,V)=\left\lbrace x\in\mathbb{R}_+, \ V(x)<R\right\rbrace$.
\begin{lemma}\label{lemma:Lyapunov_min}We have the following:
\begin{enumerate}
\item There exists $d(\varepsilon)>0$ such that for all $0\leq s\leq t$ and $x\in\mathbb{R}_+^{*}$ we have
\begin{align*}
\mathcal{A}_s^{(t)}V(x)\leq -aV(x)+d(\varepsilon).
\end{align*}
\item For all $R>2d(\varepsilon)a^{-1}$, for all $r<s\leq t$, there exists $\alpha_{s-r}>0$ such that for all Borel set $A$ of $\mathbb{R}_+$,
\begin{align*}
\inf_{x\in B(R,V)}\mathbb{P}\left(Y_s^{(t)}\in A\big| Y_r^{(t)}=x\right)\geq \alpha_{s-r} \nu_{r,s}(A).
\end{align*}
\end{enumerate}
\end{lemma}
\begin{proof}
We first prove that $V(x)=\frac{1}{x}+x$ satisfies the first point of Lemma \ref{lemma:Lyapunov_min}. Let us compute $\mathcal{A}_s^{(t)}V_1(x)$ where $V_1(x)=x$. We have for $x\in\mathbb{R}_+$,
\begin{align*}
\mathcal{A}_s^{(t)}V_1(x)&=ax+\frac{2}{1-2\varepsilon}\varphi(s)\int_{\varepsilon x}^{(1-\varepsilon)x} (y-x)\frac{1+y\phi(s,t)}{1+x\phi(s,t)}dy\\
&=ax-\varphi(s)x^2+\frac{2}{3}\varphi(s)(\varepsilon^2-\varepsilon+1)x^2\left(1-\frac{1}{1+x\phi(s,t)}\right).
\end{align*}
Then, we obtain
\begin{align*}
\mathcal{A}_s^{(t)}V_1(x)\leq a\left(1-\frac{1}{3a}\varphi(s)(1+2\varepsilon-2\varepsilon^2)x\right)x\leq -ax+\frac{3a^2}{\varphi_1(1+2\varepsilon-2\varepsilon^2)}.
\end{align*}
Next, let $V_2(x)=\frac{1}{x}$. We have
\begin{align*}
\mathcal{A}_s^{(t)}V_2(x)&=-\frac{a}{x}+\frac{2}{1-2\varepsilon}\varphi(s)\int_{\varepsilon x}^{(1-\varepsilon)x}\left(\frac{1}{y}-\frac{1}{x}\right)\frac{1+y\phi(s,t)}{1+x\phi(s,t)}dy.
\end{align*}
Using that for all $x\geq 0$ and $y\in[\varepsilon x,(1-\varepsilon)x]$, $1+y\phi(s,t)\leq 1+x\phi(s,t)$, we get
\begin{align*}
\mathcal{A}_s^{(t)}V_2(x)&\leq -\frac{a}{x}+2\varphi(s)C(\varepsilon),
\end{align*}
where $C(\varepsilon)=\frac{1}{1-2\varepsilon}\left[\log\left(\frac{1-\varepsilon}{\varepsilon}\right)-(1-2\varepsilon)\right]$.  
Noticing that $C(\varepsilon)>0$ because $\varepsilon<\frac{1}{2}$ yields
\begin{align}\label{eq:V2}
\mathcal{A}_s^{(t)}V_2(x)\leq -aV_2(x)+2\varphi_2C(\varepsilon).
\end{align}
Finally
\begin{align*}
\mathcal{A}_s^{(t)}V(x)\leq -aV(x)+d(\varepsilon),
\end{align*}
where 
\begin{align*}
d(\varepsilon)=2\varphi_2C(\varepsilon)+\frac{3a^2}{\varphi_1(1+2\varepsilon-2\varepsilon^2)}.
\end{align*}
Next, we prove the second point of Lemma \ref{lemma:Lyapunov_min}. Let us describe the shape of the subset $B(R,V)$ of $\mathbb{R}_+$ that we will consider. For all $R>2d(\varepsilon)a^{-1}$, we have
\begin{align}\label{eq:root}
B(R,V)=\left\lbrace x\in\mathbb{R}_+, \ V(x)<R\right\rbrace=\left\lbrace x_1(R)<x<x_2(R)\right\rbrace,
\end{align}
where
\begin{align*}
x_1(R)=\frac{R-\sqrt{R^2-4}}{2},\ x_2(R)=\frac{R+\sqrt{R^2-4}}{2}.
\end{align*}
Now, we prove the second point.  
Let $R>2d(\varepsilon)a^{-1}$, $x\in B(R,V)$ and let $A$ be a Borel set. Let $n\in\mathbb{N}$ be such that
\begin{align}\label{eq:nnu}
\left(\frac{1-\varepsilon}{\varepsilon}\right)^{n-1} >\frac{x_2(R)}{x_1(R)}.
\end{align} Let $0\leq r<s \leq t$. Considering the case where the auxiliary process jumped exactly $n$ times between $r$ and $s$, we have
\begin{align*}
P_{r,s}^{(t)}\left(x,A\right)\geq \mathbb{E}\left[\mathbf{1}_{\left\{Y_s^{(t)}\in A\right\}}\mathbf{1}_{\left\{r\leq\tau_1\leq s\right\}}\mathbf{1}_{\left\{\tau_1\leq\tau_2\leq s\right\}}\ldots\mathbf{1}_{\left\{\tau_{n-1}\leq\tau_n\leq s\right\}}\mathbf{1}_{\left\{\tau_{n+1}\geq s\right\}}\middle| Y_r^{(t)}=x\right],
\end{align*}
where $\tau_i$ denotes the time of the $i$th jump of the auxiliary process, $i=1,\ldots,n$. Let us denote by $\mathcal{F}_s^{(t)}$ the filtration generated by the auxiliary process $(Y_s^{(t)},s\leq t)$ up to time $s$. Conditioning with respect to $\mathcal{F}_{\tau_1}^{(t)}$ and using the strong Markov property and the fact that between two jumps, the growth of the auxiliary process is exponential at rate $a$, we get
\begin{align*}
P_{r,s}^{(t)}\left(x,A\right)\geq\mathbb{E}\Bigg[\mathbf{1}_{\left\{r\leq\tau_1\leq s\right\}}\int_{J_{r,\tau_1}(x)}\mathbb{E}\Bigg[\mathbf{1}_{\left\{\tau_1\leq\tau_2\leq s\right\}}\ldots\mathbf{1}_{\left\{\tau_{n-1}\leq\tau_n\leq s\right\}}&\mathbf{1}_{\left\{ \tau_{n+1}\geq s\right\}}\mathbf{1}_{\left\{Y_s^{(t)}\in A\right\}}\Bigg|Y_{\tau_1}^{(t)}=y_1\Bigg]\\
&\times\widehat{Q}_{\tau_1}^{(t)}\left(xe^{a(\tau_1-r)},y_1\right)dy_1\Bigg| Y_r^{(t)}=x\Bigg].
\end{align*}
where for all $r\leq s\leq t$ and $x\in\mathcal{X}$, we set $J_{r,s}(x)=\left[\varepsilon xe^{a(s-r)};(1-\varepsilon)xe^{a(s-r)}\right]$. Introducing the probability density function of the first division time $\tau_1$ yields
\begin{align*}
P_{r,s}^{(t)}\left(x,A\right)\geq \int_r^s g_{r}^{(t)}(x,t_1)\int_{J_{r,t_1}(x)}\mathbb{E}\Bigg[\mathbf{1}_{\left\{t_1\leq\tau_2\leq s\right\}}\ldots\mathbf{1}_{\left\{\tau_{n-1}\leq\tau_n\leq s\right\}}&\mathbf{1}_{\left\{ \tau_{n+1}\geq s\right\}}\mathbf{1}_{\left\{Y_s^{(t)}\in A\right\}}\Bigg|Y_{t_1}^{(t)}=y_1\Bigg]\\
&\times\widehat{Q}_{t_1}^{(t)}\left(xe^{a(t_1-r)},y_1\right)dy_1,
\end{align*}
where for all $r\leq s\leq t$ and $x\in\mathcal{X}$, $$g_{r}^{(t)}(x,s)=\widehat{B}_{s}^{(t)}\left(xe^{a(s-r)}\right)\exp\left(-\int_r^{s} \widehat{B}_{u}^{(t)}\left(xe^{a(u-r)}\right)du\right).$$
Using the same argument iteratively, we get
\begin{align*}
P_{r,s}^{(t)}\left(x,A\right)\geq \int_{E_0} g_{r}^{(t)}(x,t_1)&\int_{E_1} g_{t_1}^{(t)}(y_1,t_2)\ldots\int_{E_{n-1}} g_{t_{n-1}}^{(t)}(y_{n-1},t_n)e^{-\int_{t_n}^{s} \widehat{B}_{u}^{(t)}\left(y_ne^{a(u-t_n)}\right)du}\\
&\times\mathbf{1}_{\left\{y_ne^{a(s-t_n)}\in A\right\}}\prod_{i=0}^{n-1}\widehat{Q}_{t_{i+1}}^{(t)}\left(y_{i}e^{a(t_{i+1}-t_i)},dy_{i+1}\right)dt_n\times\ldots\times dt_1,
\end{align*}	
where $y_0=x$ and $t_0=r$ and $E_i = [t_i,s]\times J_{t_{i},t_{i+1}}(y_{i})$, for $i=0,\ldots, n-1$. 
Next, since $x\mapsto\widehat{B}_s^{(t)}(x)$ is increasing, we have
\begin{align*}
\prod_{i=0}^n\exp\left(-\int_{t_i}^{t_{i+1}} \widehat{B}_{u}^{(t)}\left(y_ie^{a(u-t_i)}\right)du\right)&\geq \exp\left(-\int_r^s\widehat{B}_u^{(t)}\left(xe^{a(u-r)}\right)du\right)\geq e^{-2\varphi_2a^{-1}x_2(R)\left(e^{a(s-r)}-1\right)},
\end{align*}
where $t_{n+1}=s$. 
Noticing that
\begin{align*}
\widehat{B}_s^{(t)}(x)\geq x\varphi_1, \ \widehat{Q}_s^{(t)}(x,y)\geq \frac{2\varepsilon}{x(1-2\varepsilon)},
\end{align*}
yields
\begin{align*}
P_{r,s}^{(t)}\left(x,A\right)\geq & C_{r,s}\int_{\mathcal{E}_{n-2}}\left(\int_{t_{n-1}}^s\left(\int_{J_{t_{n-1},t_n}(y_{n-1})}\mathbf{1}_{\left\{y_ne^{a(s-t_n)}\in A\right\}}dy_n\right)dt_n\right)dy_{n-1}dt_{n-1}\ldots dy_1dt_1,
\end{align*}	
where $\mathcal{E}_{n-2}=E_0\times \ldots\times E_{n-2}$ and
\begin{align*}
C_{r,s}=\exp\left(-2\varphi_2\frac{x_2(R)}{a}\left(e^{a(s-r)}-1\right)\right)\left(\frac{2\varphi_1\varepsilon}{1-2\varepsilon}\right)^n.
\end{align*}
Applying the substitution $z = y_ne^{a(s-t_n)}$, we get
\begin{align*}
P_{r,s}^{(t)}\left(x,A\right)\geq & C_{r,s}I^{(n)}_{r,s}(x,A), 
\end{align*}	
where
\begin{align*}
I^{(n)}_{r,s}(x,A)=\frac{1}{a}\int_{\mathcal{E}_{n-2}}\left(1-e^{-a(s-t_{n-1})}\right)\left(\int_{J_{t_{n-1},s}(y_{n-1})}\mathbf{1}_{\left\{ z\in A\right\}}dz\right)dy_{n-1}dt_{n-1}\ldots dy_1dt_1.
\end{align*}
%
Let $0<\delta_1<1<\delta_2$ be such that 
\begin{align}\label{eq:delta}
\left(\frac{\delta_1}{\delta_2}\right)^{n-1}\geq \frac{\varepsilon}{1-\varepsilon}.
\end{align}
We prove the following proposition by induction for $1\leq k\leq n$: there exists $C>0$ depending on $\varepsilon$, $k$, $\delta_1$, $\delta_2$ and $a$ such that
\begin{align*}
 I^{(k)}_{r,s}(x,A)\geq C\left(1-e^{-a(s-r)}\right)^kx^{k-1} \int_{\delta_2^{k-1}\varepsilon^n x e^{a(s-r)}}^{\delta_1^{k-1}(1-\varepsilon)^k xe^{a(s-r)}}\mathbf{1}_{z\in A}dz.
\end{align*}
The verification for $k=1$ is straightforward. We assume now that the proposition is satisfied for $k-1$, for some $k\in\llbracket 2, n\rrbracket$. Then, there exists $C>0$ such that
\begin{align*}
I^{(k)}_{r,s}(x,A)&=\int_r^s\int_{J_{r,t_1}(x)}I_{t_1,s}^{(k-1)}(y_1,A)dy_1dt_1\\
&\geq C\int_r^s\int_{J_{r,t_1}(x)}y_1^{k-2}\left(1-e^{-a(s-t_1)}\right)^{k-1}\int_{\delta_2^{k-2}\varepsilon^{k-1}y_1e^{a(s-t_1)}}^{\delta_1^{k-2}(1-\varepsilon)^{k-1}y_1e^{a(s-t_1)}}\mathbf{1}_{\left\{z\in A\right\}}dzdy_1dt_1.
\end{align*}
Switching the integrals and using that $y_1>\varepsilon xe^{a(t_1-r)}$, we get
\begin{align*}
I^{(k)}_{r,s}(x,A)
&\geq C\int_r^s\left(1-e^{-a(s-t_1)}\right)^{k-1}x^{k-2}e^{a(k-2)(t_1-r)}\left(I_1+I_2+I_3\right)dt_1,
\end{align*}
where
\begin{align*}
I_1 &= \int_{\delta_2^{k-2}\varepsilon^{k}xe^{a(s-r)}}^{\delta_2^{k-2}(1-\varepsilon)\varepsilon^{k-1}xe^{a(s-r)}}\mathbf{1}_{\left\{z\in A\right\}}dz\left(\frac{z}{\delta_2^{k-2}\varepsilon^{k-1}}e^{-a(s-t_1)}-\varepsilon xe^{a(t_1-r)}\right),\\
I_2 &= \int_{\delta_2^{k-2}(1-\varepsilon)\varepsilon^{k-1}xe^{a(s-r)}}^{\delta_1^{k-2}(1-\varepsilon)^{k-1}\varepsilon xe^{a(s-r)}}\mathbf{1}_{\left\{z\in A\right\}}dz(1-2\varepsilon) xe^{a(t_1-r)},\\
I_3 &= \int_{\delta_1^{k-2}(1-\varepsilon)^{k-1}\varepsilon xe^{a(s-r)}}^{\delta_1^{k-2}(1-\varepsilon)^{k}xe^{a(s-r)}}\mathbf{1}_{\left\{z\in A\right\}}dz\left((1-\varepsilon)xe^{a(t_1-r)}-\frac{z}{\delta_1^{k-2}(1-\varepsilon)^{k-1}}e^{-a(s-t_1)}\right). 
\end{align*}
Next, reducing the intervals of integration for $I_1$ and $I_3$ and using that $\delta_2\varepsilon\leq(1-\varepsilon)$ and $\delta_1(1-\varepsilon)\geq \varepsilon$ according to \eqref{eq:delta}, we obtain,
\begin{align*}
&I_1\geq \int_{\delta_2^{k-1}\varepsilon^{k}xe^{a(s-r)}}^{\delta_2^{k-2}(1-\varepsilon)\varepsilon^{k-1}xe^{a(s-r)}}\mathbf{1}_{\left\{z\in A\right\}}dz(\delta_2-1)\varepsilon xe^{a(t_1-r)},\\
&I_3 \geq \int_{\delta_1^{k-2}(1-\varepsilon)^{k-1}\varepsilon xe^{a(s-r)}}^{\delta_1^{k-1}(1-\varepsilon)^{k}xe^{a(s-r)}}\mathbf{1}_{\left\{z\in A\right\}}dz(1-\delta_1)(1-\varepsilon)xe^{a(t_1-r)}. 
\end{align*}
Therefore, gathering the three integrals and integrating with respect to $t_1$, we get
\begin{align*}
I_{r,s}^{(k)}(x,A)\geq C\left(1-e^{-a(s-r)}\right)^kx^{k-1}\int_{\delta_2^{k-1}\varepsilon^{k}xe^{a(s-r)}}^{\delta_1^{k-1}(1-\varepsilon)^{k}xe^{a(s-r)}}\mathbf{1}_{\left\{z\in A\right\}}dz,
\end{align*}
where the constant $C$ varies from line to line and the proposition holds at stage $k$.
Finally, we have
\begin{align*}
P_{r,s}^{(t)}(x,A)\geq \alpha_{s-r}\nu_{r,s}(A),
\end{align*}
where
\begin{align*}
\alpha_{s-r}=&C\left(1-e^{-a(s-r)}\right)^nx_1(R)^{n-1}e^{a(s-r)}\left(\delta_1^{n-1}(1-\varepsilon)^nx_1(R)-\delta_2^{n-1}\varepsilon^nx_2(R)\right)C_{r,s},\\
\nu_{r,s}(A)=&\frac{1}{e^{a(s-r)}\left(\delta_1^{n-1}(1-\varepsilon)^nx_1(R)-\delta_2^{n-1}\varepsilon^nx_2(R)\right)}\int_{\delta_2^{n-1}\varepsilon^{n}xe^{a(s-r)}}^{\delta_1^{n-1}(1-\varepsilon)^{n}xe^{a(s-r)}}\mathbf{1}_{\left\{z\in A\right\}}dz,
\end{align*}
and
\begin{align*}
\delta_1^{n-1}(1-\varepsilon)^nx_1(R)-\delta_2^{n-1}\varepsilon^nx_2(R)>0,
\end{align*}
according to \eqref{eq:nnu} and \eqref{eq:delta}.
\end{proof}
Next, we check that Assumption \ref{assu:J2phiNt} is satisfied. The verification of the first point is straightforward as
$$\frac{\log(m(x,0,t))}{t}=\frac{\log(1+x\int_0^t \varphi(r)e^{ar}dr)}{t}\xrightarrow[t\rightarrow +\infty]{}a.$$
To check the second point of Assumption \ref{assu:J2phiNt}, we prove that the moments of the auxiliary process are bounded. For all $p\in\mathbb{N}^*$, $0\leq s\leq t$ and $x\geq 0$, we denote by $$f_{p}^{(t)}(x,s)=\E_x\left[\left(Y_{s}^{(t)}\right)^p\right].$$
\begin{lemma}\label{lemma:moment}
For all $p\in\mathbb{N}^*\bigcup \left\lbrace -1\right\rbrace$ and $x\geq 0$, we have
\begin{align*}
\sup_{t\geq 0}\sup_{s\leq t}\E_x\left[\left(Y_{s}^{(t)}\right)^p\right]<+\infty.
\end{align*}
\end{lemma}
\begin{rem}
The moments that we need to control in order to check the second point of Assumption \ref{assu:J2phiNt} depend on the function $V$. The shape of the Lyapunov function $V(x)=x+x^{-1}$ was convenient for the proof of the second point of Lemma \ref{lemma:Lyapunov_min}. Indeed, the proof relies on the fact that $B(R,V)$ is lower bounded by a positive real number. This is the case because of the term $x^{-1}$ in $V$. Because of this term, we need to control the first harmonic moment of the auxiliary process.
\end{rem}
\begin{proof}
Let $p\in\mathbb{N}^*$ be a positive integer. We have, using \eqref{eq:gene_ex} and Dynkin's formula,
\begin{align*}
f_p^{(t)}(x,s)=& x^p+ap\int_0^{s} f_p^{(t)}(x,r)dr\\
&+2\int_0^{s} \E_x\left[\varphi(r)\int_{\varepsilon Y_r^{(t)}}^{(1-\varepsilon)Y_r^{(t)}}\left(y^p-\left(Y_{r}^{(t)}\right)^p\right)\frac{1+y\phi(r,t)}{1+Y_r^{(t)}\phi(r,t)}dy\right]\frac{dr}{1-2\varepsilon}.
\end{align*}
By differentiation with respect to $s$ of the last equality we get
\begin{align*}
\partial_s f_{p}^{(t)}(x,s)= apf_{p}^{(t)}(x,s)+2\E_x\left[\varphi(s)\int_{\varepsilon Y_s^{(t)}}^{(1-\varepsilon)Y_s^{(t)}}\left(y^p-\left(Y_{s}^{(t)}\right)^p\right)\frac{1+y\phi(s,t)}{1+Y_s^{(t)}\phi(s,t)}\frac{dy}{1-2\varepsilon}\right].
\end{align*}
Next, we notice that for $\varepsilon x\leq y\leq (1-\varepsilon)x$, we have
\begin{align*}
\frac{m(y,s,t)}{m(x,s,t)}\geq \frac{1+\varepsilon x\phi(s,t)}{1+x\phi(s,t)}\geq \varepsilon.
\end{align*}
Then
\begin{align*}
\partial_s f_{p}^{(t)}(x,s)&\leq apf_{p}^{(t)}(x,s)+2\varepsilon\E_x\left[\varphi(s)\int_{\varepsilon Y_s^{(t)}}^{(1-\varepsilon)Y_s^{(t)}}\left(y^p-\left(Y_{s}^{(t)}\right)^p\right)\frac{dy}{1-2\varepsilon}\right]\\
&\leq apf_{p}^{(t)}(x,s)-C(\varepsilon)f_{p+1}^{(t)}(x,s),
\end{align*}
where $C(\varepsilon):=\frac{2\varepsilon}{1-2\varepsilon}\varphi_1\left(1-2\varepsilon-\frac{(1-\varepsilon)^{p+1}-\varepsilon^{p+1}}{p+1}\right)$. Moreover, $C(\varepsilon)>0$ because $\varepsilon<\frac{1}{2}$. Applying Jensen inequality, we have $f_{p+1}^{(t)}(s)\geq f_{p}^{(t)}(s)^{1+1/p}$. Finally, we obtain the following differential inequality:
\begin{align*}
\partial_s f_{p}^{(t)}(x,s)\leq& F\left(f_{p}^{(t)}(x,s)\right),
\end{align*}
where $F(x)=apx-C(\varepsilon)x^{1+1/p}$ for all $x\geq 0$. We notice that there exists $x_0>0$ such that $F> 0$ on $(0,x_0)$ and $F< 0$ on $(x_0,+\infty)$. Then, any solution to the equation $y'=F(y)$ is bounded by $y(0)\vee x_0$ and so is $f_{p}^{(t)}(x,\cdot)$.

Next, we prove that the first harmonic moment of the auxiliary process is bounded. 
Let us recall that $V_2(x)=1/x$. Let $x\in\mathcal{X}$ and $0\leq s\leq t$.  According to Kolmogorov's forward equation, we have
\begin{align*}
\partial_s P_{0,s}^{(t)}V_2(x)=P_{0,s}^{(t)}\mathcal{A}_s^{(t)}V_2(x).
\end{align*}
Using \eqref{eq:V2}, we get
\begin{align*}
\partial_sP_{0,s}^{(t)}V_2(x)\leq -aP_{0,s}^{(t)}V_2(x)+2\varphi_2 C(\varepsilon).
\end{align*}
Finally, using Gr\"onwall's inequality, we obtain
\begin{align*}
P_{0,s}^{(t)}V_2(x)\leq \left(\frac{1}{x}-\frac{2\varphi_2 C(\varepsilon)}{a}\right)e^{-as}+\frac{2\varphi_2 C(\varepsilon)}{a}.
\end{align*}
\end{proof}
\begin{lemma}\label{lemma:verif_J2Nt}
For all $x\geq 0$, we have
\begin{align*}
\sup_{t\geq 0}\mathbb{E}_x\left[B\left(t,Y_t^{(t)}\right)J\left(1\vee V(\cdot)\varphi_t(x,\cdot)\right)\left(Y_t^{(t)}\right)\right]<+\infty.
\end{align*}
\end{lemma}
\begin{proof}
In our case, $1\vee V(x)=V(x)$. First, we have for all $x\in\mathbb{R}_+$ and all $s,t\in\mathbb{R}_+$ with $s\leq t$,
\begin{align*}
1+\frac{x}{a}\varphi_1(e^{a(t-s)}-1)\leq m(x,s,t)\leq 1+\frac{x}{a}\varphi_2(e^{a(t-s)}-1).	
\end{align*}
Then, for all $x,y\in\mathcal{X}$, we obtain
\begin{align*}
\varphi_t(x,y)=\sup_{r\geq t}\frac{m(x,0,t)m(y,t,r)}{m(x,0,r)}\leq \sup_{r\geq t}\frac{\left(1+\frac{x}{a}\varphi_2e^{at}\right)\left(1+\frac{y}{a} \varphi_2e^{a(r-t)}\right)}{1+\frac{x}{a}\varphi_1\left(e^{ar}-1\right)}\leq\frac{\left(1+\frac{x}{a}\varphi_2\right)\left(1+\frac{y}{a}\varphi_2\right)}{\frac{x}{a}\varphi_1\wedge 1}.
\end{align*}
Next, for all $\theta\in (0,1)$, we have
\begin{align*}
\varphi_t\left(x,\theta y\right)\varphi_t\left(x,(1-\theta)y\right)\leq  \left(\varphi_t(x,y)\right)^2\leq A_1(x)A_2(y),
\end{align*}
where
\begin{align*}
A_1(x)=\left(\frac{x}{a}\varphi_1\wedge 1\right)^{-2}\left(1+\frac{x}{a}\varphi_2\right)^2,\ A_2(y)=\left(1+\frac{y}{a}\varphi_2\right)^2.
\end{align*}
Moreover, for $\theta\in [\varepsilon,1-\varepsilon]$ and for all $x\in\mathcal{X}$, $V(\theta x)V((1-\theta)x)\leq (\varepsilon x)^{-2}+x^2+2\varepsilon^{-1}$. Then,
\begin{align*}
J\left(V(\cdot)\varphi_t(x,\cdot)\right)\left(y\right)&\leq 2 \int_{\varepsilon}^{1-\varepsilon} V\left(\theta y\right) V\left((1-\theta) y\right)\varphi_t\left(x,\theta y\right)\varphi_t\left(x,(1-\theta)y\right)\frac{d\theta}{1-2\varepsilon}\\
&\leq \left((\varepsilon y)^{-2}+y^2+2\varepsilon^{-1}\right)A_1(x)A_2(y)\leq A_1(x)\sum_{k=0}^6C_k(\varepsilon)y^{k-2},
\end{align*}
where for all $k=0\ldots 6$ $C_k(\varepsilon)$ are constants depending on $x,a,\varepsilon, \varphi_2$. Then, we get
\begin{align*}
\mathbb{E}_x\left[B\left(t,Y_t^{(t)}\right)J\left(V(\cdot)\varphi_t(x,\cdot)\right)\left(Y_t^{(t)}\right)\right] \leq  \varphi_2A_1(x)\sum_{k=0}^6 C_k(\varepsilon)\sup_{t\geq 0}\mathbb{E}_x\left[\left(Y_t^{(t)}\right)^{k-1}\right]<\infty,
\end{align*}
according to Lemma \ref{lemma:moment}.
\end{proof}
Last, we verify that Assumption \ref{assu:Ntmt} is satisfied.
\begin{lemma}\label{lemma:varWt}
For all $t\geq 0$, $x\in\mathcal{X}$, we have
\begin{align*}
\mathbb{E}_{\delta_{x_0}}\left[\left(\frac{N_t}{m(x_0,0,t)}\right)^2\right]\leq \frac{a^2+\varphi_2 x\left(a+2\varphi_2 x\right)+\varphi_2^2x^2}{\left(\min(a,\varphi_1 x)\right)^2}.
\end{align*}
\end{lemma}
\begin{proof}
According to It\^o's formula, we have, for all $x\in\mathcal{X}$ and $t\geq 0$,
\begin{align*}
\mathbb{E}_{\delta_x}\left[N_t^2\right]=1+x\int_0^t \varphi(s)e^{as}\left(2\mathbb{E}_{\delta_x}\left[N_s\right]+1\right)ds. 
\end{align*}After some calculations, we obtain
\begin{align*}
\mathbb{E}_{\delta_x}\left[N_t^2\right]\leq \frac{e^{2at}}{a^2}\left(a^2+\varphi_2x\left(a+2\varphi_2x\right)+\varphi_2^2 x^2\right)
\end{align*}	
Moreover, we have
\begin{align*}
m(x,0,t)^2\geq e^{2at}\left(e^{-at}+\frac{x}{a}\varphi_1(1-e^{-at})\right)^2\geq e^{2at}\left(\min\left(1, \frac{x}{a}\varphi_1\right)\right)^2,
\end{align*}
and the result follows.
\end{proof}
\paragraph*{Acknowledgment}I would like to thank Vincent Bansaye for his continual guidance during this work. I acknowledge partial support by the Chaire Mod\'elisation Math\'ematique et Biodiversit\'e of Veolia Environnement - \'Ecole Polytechnique - Museum National Histoire Naturelle - F.X. This work is supported by the "IDI 2014" project funded by the IDEX Paris-Saclay, ANR-11-IDEX-0003-02 and by the French national research agency (ANR) via project MEMIP (ANR-16-CE33-0018).
\bibliographystyle{alpha}
\bibliography{bibli}
\end{document}